\documentclass[11pt]{article}
\selectfont
\usepackage[top=2.54cm, bottom=2.54cm, left=2.5cm, right=2.5cm]{geometry}
\usepackage[tbtags]{amsmath}
\usepackage{amssymb}
\usepackage{amsthm}
\usepackage{appendix}
\usepackage{latexsym}
\usepackage{mathrsfs}
\usepackage{xcolor}
\usepackage{wasysym}
\usepackage{fancyhdr}
\usepackage[colorlinks=true, linkcolor=red, citecolor=blue, urlcolor=magenta]{hyperref}
\usepackage{float}
\usepackage{graphicx}
\usepackage[numbers,sort&compress]{natbib}

\usepackage{enumerate}
\usepackage{verbatim}

\allowdisplaybreaks[4]

\newcommand{\R}{\mathcal R}
\newcommand{\B}{\mathcal B}

\newtheorem{theorem}{Theorem}[section]
\newtheorem{lemma}[theorem]{Lemma}
\newtheorem{proposition}[theorem]{Proposition}
\newtheorem{corollary}[theorem]{Corollary}
\theoremstyle{definition}

\newtheorem{claim}{Claim}
\newtheorem{problem}[theorem]{Problem}
\newtheorem{conjecture}[theorem]{Conjecture}

\title{The Ramsey threshold for trees versus odd cycles}
\author{Qizhong Lin\footnote{Center for Discrete Mathematics, Fuzhou University,
Fuzhou 350108, P.~R.~China. Email: {\tt linqizhong@fzu.edu.cn}. Supported in part by the National Key R\&D Program of China (Grant No. 2023YFA1010202) and the NSFC (No.\ 12571361).}
\;\; and\;\; Chunlin You\footnote{Corresponding author. School of Mathematics and Statistics, Yancheng Teachers University, Yancheng 224002, P.~R.~China. E-mail: {\tt chunlin\_you@163.com}.
			Supported by the National Natural Science Foundation of China
			(No. 12401469) and the Young Elite Scientist Sponsorship Program by JSAST (JSTJ-2025-892).}}
\date{}

\begin{document}

\maketitle

\begin{abstract}
A longstanding problem of Burr, Erd\H{o}s, Faudree, Rousseau and Schelp (\emph{Trans. Amer. Math. Soc.}, 1982) is to determine, for each odd $m\ge3$, the least integer $f(m)$ such that every tree $T_n$ on $n\ge f(m)$ vertices satisfies $R(T_n,C_m)=2n-1$. We settle this problem for all sufficiently large odd $m$. More precisely, we establish
\[
f(m)=\left\lceil \frac{2m-1}{3} \right\rceil
\]
for all such $m$, where the lower bound follows from a result of Faudree,
Lawrence, Parsons and Schelp. This also confirms a conjecture of Huang, Zhang and Chen for all such $m$.
\end{abstract}

\section{Introduction}

For graphs $G$ and $H$, the Ramsey number $R(G,H)$ is the least $N$
such that every red--blue coloring of the complete graph $K_N$ contains either a red copy of
$G$ or a blue copy of $H$. These numbers lie at the heart of Ramsey
theory, a central branch of combinatorics with a long and rich history.

Let $T_n$, $P_n$, and $C_n$ denote the tree, path, and cycle with $n$ vertices, respectively.
Consider the coloring of $K_{2n-2}$ formed by two disjoint red cliques of order $n-1$, with every
edge between the two cliques colored blue. Neither red component contains
$T_n$, and the blue graph is bipartite and hence contains no odd cycle.
Therefore, we obtain the natural lower bound:
\begin{equation}
                         R(T_n,C_m)\ge2n-1.                 \label{eq:lower}
\end{equation}

In 1982, Burr, Erd\H{o}s, Faudree, Rousseau and Schelp \cite{BurrEtAl1982} posed the following problem, asking when equality holds in \eqref{eq:lower}:

\begin{problem}[Burr, Erd\H{o}s, Faudree, Rousseau and Schelp \cite{BurrEtAl1982}]
\label{pb:be}
For every odd $m\geq 3$, what is the least integer $f(m)$ such that $R(T_n,C_m)=2n-1$ whenever $n\geq f(m)$?
\footnote{This problem also appears at \url{https://www.erdosproblems.com/forum/thread/Missing\%20Erdos\%20problems}.}
\end{problem}

It is known \cite{cha} that $R(T_n,C_3)=2n-1$ for all $n\ge1$; thus the case $m=3$ is already covered.

A necessary lower bound on $f(m)$ is obtained by taking $T_n$ to be a path. Faudree,
Lawrence, Parsons and Schelp \cite{FaudreeEtAl1974} proved that
\[
R(P_n,C_m)=\max\left\{2n-1,\,m-1+\left\lfloor n/2\right\rfloor\right\}
\qquad(2\le n\le m,\ m\text{ odd}).
\]
If $n \le 2(m-1)/3$, then, since $\lfloor n/2 \rfloor > n/2 - 1$, we have $m-1 + \lfloor n/2 \rfloor > 2n-1$; thus the second term dominates. This implies
\begin{equation}
 f(m)\ge
 \left\lceil\frac{2m-1}{3}\right\rceil.                  \label{eq:path-lower}
\end{equation}

In the same paper, Burr, Erd\H{o}s, Faudree, Rousseau and Schelp \cite{BurrEtAl1982} gave the first general bound $f(m)\le756m^{10}$. This bound was subsequently improved through a series of substantial advances. Brennan \cite{Brennan2016} initiated the modern approach by obtaining the first linear estimate $f(m)\le25m$, via a structural reduction from the Ramsey question to the existence of a large vertex subset with prescribed degree properties. Fan and Lin \cite{FanLin2025} refined this framework with a strengthened dichotomy argument and a reconstruction of the end-edge matching, lowering the coefficient to $4$. Most recently, Huang, Zhang and Chen \cite{HuangZhangChen2026} further improved the bound to $2m-4$ via a structural argument based on a blue \(K_{2,n}\)-forcing technique.

Nevertheless, the exact value of $f(m)$ remained unknown. A substantial gap of $4/3$ in the leading coefficient still separated the best known upper bound from the lower bound $\lceil \frac{2m-1}3\rceil$. This persistent gap led Huang, Zhang and Chen \cite{HuangZhangChen2026} to conjecture that the path obstruction is the only obstruction, i.e., that equality holds for every odd $m\ge5$.

\begin{conjecture}[Huang, Zhang and Chen~\cite{HuangZhangChen2026}]
\label{conj:exact-threshold}
For every odd integer $m\ge5$,
$
                         f(m)=\left\lceil\frac{2m-1}{3}\right\rceil.
$
\end{conjecture}

The following theorem establishes that the lower bound \eqref{eq:path-lower} is the true threshold for all sufficiently large odd $m$.

\begin{theorem}\label{thm:main}
There is an absolute constant $C$ such that the following holds.  Let $m\ge5$
be odd, and suppose that
\[
 n\ge
 \max\left\{
   C,
   \left\lceil\frac{2m-1}{3}\right\rceil
 \right\}.
\]
Then, for every $n$-vertex tree $T_n$,
$R(T_n,C_m)=2n-1$.
\end{theorem}

Since $C\le \lceil \frac{2m-1}3\rceil$ for all sufficiently large odd $m$, Theorem~\ref{thm:main} gives $f(m)\le \lceil \frac{2m-1}3\rceil$. Together with the lower bound \eqref{eq:path-lower}, this yields the following corollary.

\begin{corollary}\label{cor:f(m)}
For all sufficiently large odd $m$,
$f(m)=\left\lceil\frac{2m-1}{3}\right\rceil.$
\end{corollary}

For all sufficiently large odd $m$, the corollary resolves Problem~\ref{pb:be} of Burr, Erd\H{o}s, Faudree, Rousseau and Schelp, which had remained open since 1982. It also confirms Conjecture~\ref{conj:exact-threshold} of Huang, Zhang and Chen in this range.

The proof of Theorem~\ref{thm:main} combines the dense
Loebl--Koml\'os--S\'os theorem with a cycle-extension and closure
argument designed to preserve both a prescribed set of high-degree
vertices and a lower bound on the cycle length. The main difficulty
lies in the long-cycle range \(m\ge n+2\), which is beyond the reach
of the median-degree pancyclicity lemma; our treatment of this range
constitutes the main new contribution of the proof. To overcome this difficulty,
we first construct an auxiliary linear forest from a maximal family
of blue ears, then apply a new tree-embedding lemma to place a large
subtree in the red graph, and finally use Hall's theorem to embed the
remaining leaves. For the short-cycle range \(m\le n+1\), a
median-degree pancyclicity lemma suffices. Together, these tools yield
a matching upper bound for all sufficiently large odd \(m\).

The remainder of the paper is organized as follows.  Section~\ref{sec:prelim} collects the
preliminary results used in the proof.  Section~\ref{sec:median-pancyclic} proves the
median-degree pancyclicity lemma, which handles the short-cycle range $m\le n+1$. 
Section~\ref{sec:long-cycle} establishes the key proposition for the
relevant long-cycle range
$
n+2\le m\le 2n-\left\lfloor\frac n2\right\rfloor,
$
using the cycle-extension and auxiliary-forest strategy outlined above.
Section~\ref{sec:main-proof} applies these two results to their respective
ranges and thereby completes the proof of Theorem~\ref{thm:main}. Finally, we conclude
with a brief discussion and mention a few open problems.

\section{Preliminaries}\label{sec:prelim}

All graphs are finite and simple.  For a graph $G$, a vertex $v\in V(G)$, and
a set $U\subseteq V(G)$, write $G[U]$ for the subgraph induced by $U$ and
\[
 N_G(U)=\{v\in V(G)\setminus U:uv\in E(G)\text{ for some }u\in U\}.
\]
We also write $N_G(v)$ for the neighborhood of $v$, $d_G(v)=|N_G(v)|$, and
$d_U(v)=|N_G(v)\cap U|$. Let $\delta(G)$ denote the minimum degree of $G$.
The order of a graph $G$ is $|V(G)|$, and its complement is denoted by $\overline G$.
We write \(K_r\) for the complete graph of order \(r\) and \(K_{r,s}\) for the
complete bipartite graph with part orders \(r\) and \(s\); a clique is a
complete subgraph. A \emph{linear forest} is a forest whose components are paths and isolated vertices.

A cycle in a graph $G$ is \emph{spanning} if it contains every vertex of $G$.
The graph $G$ is \emph{Hamiltonian} if it
contains a spanning cycle, called a \emph{Hamilton cycle}.  Equivalently, if
$G$ has order $N$, then its vertices can be ordered as $v_1,\ldots,v_N$ so
that
\[
 v_iv_{i+1}\in E(G)\quad\text{for }1\le i<N,
 \qquad\text{and}\qquad
 v_Nv_1\in E(G).
\]
An $xy$-path is a path with endpoints $x$ and $y$; it is a Hamilton
$xy$-path if it contains all vertices of $G$.
A graph is \emph{pancyclic} if it contains a cycle of every order between
$3$ and its order.

The following theorem of Hladk\'y and Piguet~\cite{HladkyPiguet}, which is the dense case of the Loebl--Koml\'os--S\'os theorem, gives a degree condition under which a graph contains every tree with a prescribed number of edges.

\begin{lemma}[Hladk\'y and Piguet \cite{HladkyPiguet}]\label{lem:dense-lks}
For every $\alpha>0$, there exists an integer $N_0$ such that the following
holds.  If $N\ge N_0$, $\alpha N<k<N$, and an $N$-vertex graph $G$ has at
least $N/2$ vertices of degree at least $k$, then $G$ contains every tree
with $k$ edges.
\end{lemma}

The next result gives a path through all vertices whose degrees are at least
the median threshold.

\begin{lemma}[Li and Xu~\cite{LiXu}]\label{lem:li-xu}
Let $G$ be a graph of order $N$, and let $x,y$ be distinct vertices such that
$
                         d_G(x), \; d_G(y)\ge\frac{N+1}{2}.
$
Then $G$ contains an $(x,y)$-path containing every vertex $z$ with
$d_G(z)\ge(N+1)/2$.
\end{lemma}

We shall also need the following classical pancyclicity criterion for Hamiltonian graphs.

\begin{lemma}[Bondy~\cite{Bondy1971}]
\label{lem:bondy-cycle-edge}
Let $C$ be a Hamilton cycle of a graph $G$ of order $N$.
If $C$ contains two consecutive vertices $x$ and $y$ such that
$
d_G(x)+d_G(y)\ge N+1,
$
then $G$ is pancyclic.
\end{lemma}

The above pancyclicity lemma yields the following standard consequence.

\begin{corollary}[Bondy~\cite{Bondy1971}]\label{lem:bondy}
If an $N$-vertex graph $G$ satisfies $\delta(G)\ge N/2$, then either $G$ is pancyclic or $N$ is even and $G=K_{N/2,N/2}$.
\end{corollary}

\section{The short-cycle range}\label{sec:median-pancyclic}

This section is devoted to the short-cycle regime \(m\le n+1\).
We prove a pancyclicity lemma (stated for a general parameter \(q\))
which guarantees all cycle lengths up to \(q+1\);
in the main theorem this will be applied with \(q=n\),
thereby covering every desired cycle length \(m\le n+1\).

We first record a weighted counting lemma used in the proof below.

\begin{lemma}\label{lem:weighted-cover}
Let $G$ be a graph of order $q$, and let $D_1,\ldots,D_p$ be cliques of $G$,
not necessarily distinct.  If every vertex $v\in V(G)$ is contained in at
least $d_G(v)+1$ cliques $D_i$, then $p\ge q$.
\end{lemma}

\begin{proof}
Assign to each vertex $v\in V(G)$ the weight $w(v)=\frac{1}{d_G(v)+1}.$
For each vertex $v\in V(G)$, let $p(v)$ denote the number of cliques
$D_i$ containing $v$. We bound the sum
\[
\sum_{v\in V(G)}p(v)w(v)
\]
from above and below.

First, fix a clique $D_i$. For every $v\in D_i$, since $D_i$ is a clique of
$G$, we have $d_G(v)+1\ge |D_i|$; hence
$\sum_{v\in D_i}w(v)\le |D_i|\cdot \frac1{|D_i|}=1.$
Summing over all \(p\) cliques, and noting that each vertex \(v\) is counted
once for each clique containing it, we obtain
\[
\sum_{v\in V(G)} p(v)w(v)=\sum_{i=1}^{p}\sum_{v\in D_i}w(v)\le p.
\]

On the other hand, the hypothesis $p(v)\ge d_G(v)+1$ gives
$
p(v)w(v)=p(v)\cdot \frac1{d_G(v)+1}\ge 1
$
for every $v\in V(G)$. Hence \[\sum_{v\in V(G)} p(v)w(v)\ge q.\]

Thus $p\ge q$ follows.
\end{proof}

We now turn to the pancyclicity lemma itself.

\begin{lemma}\label{lem:median-pancyclic}
Let $q\ge3$, and let $G$ be a graph of order $2q-1$.  If there are $q$ vertices of $G$ with degree at least $q$, then $G$ contains a cycle $C_k$ for every integer $3\le k\le q+1$.
\end{lemma}

\begin{proof}
We first prove the following assertion.  If $S\subseteq V(G)$ satisfies
\[
|S|\ge q
\qquad\text{and}\qquad
d_G(v)\ge q\quad\text{for every }v\in S,
\]
then $G$ has a pancyclic induced subgraph whose vertex set contains $S$.

Since $|V(G)\setminus S|\le q-1$, every vertex of $S$ has a neighbor in
$S$.  Choose an edge $xy$ inside $S$.  As
$
\frac{|V(G)|+1}{2}=q,
$
Lemma~\ref{lem:li-xu} gives an $xy$-path $P$ containing every vertex of
degree at least $q$.  In particular, $S\subseteq V(P)$.  Hence
$P\cup\{xy\}$ forms a cycle $Q_0$ containing $S$.

We now extend $Q_0$.  Suppose that a cycle $Q_i$ containing $S$ has been
constructed, and set
\[
\ell_i=|V(Q_i)|,\qquad
H_i=G[V(Q_i)],\qquad
Z_i=V(G)\setminus V(Q_i).
\]
Since
\[
\ell_i\le 2q-1<2|S|,
\]
there must exist two vertices of $S$ that are consecutive on $Q_i$. Otherwise, the
successors on $Q_i$ of the vertices in $S$ would be distinct vertices of
$V(Q_i)\setminus S$, which would imply
$
|V(Q_i)\setminus S|\ge |S|
$
and hence $\ell_i\ge2|S|$, a contradiction.

Let $u_iv_i$ be an edge of $Q_i$ with $u_i,v_i\in S$.
If
$
d_{H_i}(u_i)+d_{H_i}(v_i)\ge \ell_i+1,
$
then Lemma~\ref{lem:bondy-cycle-edge}, applied to the Hamilton cycle
$Q_i$ of $H_i$, implies that $H_i$ is pancyclic, and the assertion follows.

Otherwise, we have
$
d_{H_i}(u_i)+d_{H_i}(v_i)\le \ell_i.
$
Since $d_G(u_i),d_G(v_i)\ge q$, we have
\[
\begin{aligned}
d_{Z_i}(u_i)+d_{Z_i}(v_i)
=d_G(u_i)+d_G(v_i)
  -d_{H_i}(u_i)-d_{H_i}(v_i)
\ge 2q-\ell_i
=|Z_i|+1.
\end{aligned}
\]
Therefore we have
$
N_G(u_i)\cap N_G(v_i)\cap Z_i\ne\varnothing.
$
Choose a common neighbor $z_i\in Z_i$ of $u_i$ and $v_i$.  Replacing the
edge $u_iv_i$ of $Q_i$ with the path $u_iz_iv_i$ gives a cycle
$Q_{i+1}$ satisfying
\[
S\subseteq V(Q_{i+1})
\qquad\text{and}\qquad
|V(Q_{i+1})|=\ell_i+1.
\]

Thus, whenever the degree-sum condition fails, the order of the current
cycle increases by one.  Since every cycle of $G$ has order at most
$2q-1$, only finitely many extensions are possible.

It remains to consider the case that the current cycle $Q_i$ is spanning.
Then
\[
H_i=G
\qquad\text{and}\qquad
\ell_i=|V(Q_i)|=2q-1.
\]
As above, $Q_i$ contains an edge $u_iv_i$ with $u_i,v_i\in S$.  Since
each vertex of $S$ has degree at least $q$ in $G$, we have
\[
\begin{aligned}
d_{H_i}(u_i)+d_{H_i}(v_i)
=d_G(u_i)+d_G(v_i)
\ge 2q
=\ell_i+1.
\end{aligned}
\]
Thus the degree-sum condition holds when the current cycle is spanning.
Since \(H_i = G\), Lemma~\ref{lem:bondy-cycle-edge} implies that \(G\) is pancyclic.
Hence the construction terminates with a pancyclic induced
subgraph containing \(S\).  This proves the assertion.

Now apply the assertion to the set
\[
W=\{v\in V(G):d_G(v)\ge q\}.
\]
By hypothesis, \(|W|\ge q\). The assertion yields a pancyclic induced subgraph \(G[X]\) containing \(W\). Since \(|X|\ge |W|\ge q\), it follows that \(G\) contains \(C_k\) for every \(3\le k\le q\).

It remains to prove that \(G\) contains \(C_{q+1}\). Suppose, to the contrary, that \(G\) is \(C_{q+1}\)-free. Since \(G[X]\) is pancyclic and contains no \(C_{q+1}\), we must have \(|X|\le q\). But \(|X|\ge |W|\ge q\), so \(|X|=q\) and \(X=W\). In particular, \(|W|=q\). Let \(H=G[W]\). Then \(H\) is pancyclic and hence Hamiltonian.

By Lemma~\ref{lem:li-xu}, for every edge \(uv\in E(H)\), there exists a \(uv\)-path \(P\) in \(G\) that contains every vertex of \(W\). Suppose \(P\) contains a vertex outside \(W\). Then \(P\cup\{uv\}\) is a cycle containing \(W\) of order at least \(q+1\). Applying the extension procedure from the assertion to this cycle yields a pancyclic induced subgraph of order at least \(q+1\), hence a \(C_{q+1}\), contradicting the assumption that \(G\) is \(C_{q+1}\)-free. Therefore \(V(P)=W\), so \(P\) is a Hamilton \(uv\)-path in \(H\).

Let
$U=V(G)\setminus W$.
For each $z\in U$, define $N_W(z)=N_G(z)\cap W$. We claim that each $N_W(z)$ is an independent set in $H$. Otherwise, $N_W(z)$ contains an edge $uv\in E(H)$. Let $P$ be a Hamilton $uv$-path in $H$. Then $P\cup\{zu,zv\}$ is a cycle of order $q+1$, contradicting the assumption that $G$ is $C_{q+1}$-free. Hence each $N_W(z)$ is independent in $H$ and therefore a clique in $\overline H$.

For every $w\in W$, note that $d_H(w)+d_{\overline H}(w)=q-1$, so we have
\[
d_U(w)
=d_G(w)-d_H(w)
\ge q-d_H(w)
=d_{\overline H}(w)+1.
\]
Since the family $\{N_W(z):z\in U\}$ consists of $|U|=q-1$ cliques of $\overline H$, and every vertex of $\overline H$ is covered by at least $d_{\overline H}(w)+1$ of them, Lemma~\ref{lem:weighted-cover} gives $q-1\ge q$, a contradiction. Hence \(G\) contains \(C_{q+1}\), completing the proof.
\end{proof}

\section{The long-cycle range}\label{sec:long-cycle}

\subsection{A maximal ear-forest structure}

Let \(G\) be a graph, \(W\subseteq V(G)\), and \(U=V(G)\setminus W\). A \emph{\(W\)-ear in \(G\)} is a path whose endpoints lie in \(W\) and whose internal vertices all lie in \(U\); it is \emph{nontrivial} if it has at least one internal vertex.

The following structural lemma, one of the main new ingredients of
the proof, describes a maximal family of ears when
a specified clique is not contained in any sufficiently long cycle.
It also controls how the remaining vertices attach to the resulting
linear forest.

\begin{lemma}\label{lem:max-ear-forest}
Let \(G\) be a graph, let \(W\subseteq V(G)\) induce a clique with \(|W|\ge3\), and set \(U=V(G)\setminus W\).
Let \(L_0\) be an integer satisfying \(|W|+2\le L_0\le |V(G)|\).
Suppose that \(G\) contains no cycle through \(W\) of length at least \(L_0\).
Among all families of pairwise internally vertex-disjoint nontrivial
\(W\)-ears whose endpoint pairs are distinct and form a linear forest
on \(W\), choose \(\mathcal P\) to maximize the total number of internal
vertices. Let \(F\) be the linear forest whose edges are these endpoint
pairs.

Define
\[
\rho=\sum_{P\in\mathcal P}|V(P)\cap U|,\qquad
U_{\mathrm{free}}=U\setminus\bigcup_{P\in\mathcal P}V(P),
\]
and let \(W_{\mathrm{int}}\) denote the set of vertices of degree \(2\) in \(F\).
Then the following hold:

\begin{enumerate}
\item[(i)] \(\displaystyle |U_{\mathrm{free}}|\ge |V(G)|-L_0+1.\)

\item[(ii)] For each \(u\in U_{\mathrm{free}}\), the set \(N_G(u)\cap(W\setminus W_{\mathrm{int}})\) is contained in a single path component of \(F\) and has size at most \(2\); if it has size \(2\), it is precisely the set of endpoints of that component.
\end{enumerate}
\end{lemma}

\begin{proof}
Choose \(\mathcal P\) and \(F\) as in the statement. Such a family
exists since the empty collection is admissible. For each edge
\(e=xy\in E(F)\), let \(P_e\in\mathcal P\) be the unique ear with
endpoints \(x\) and \(y\).

\medskip
(i) Suppose \(\rho\ge L_0-|W|\). Since \(W\) is a clique, we can order the components of \(F\) arbitrarily and connect them by edges of \(G[W]\) to form a Hamilton cycle of \(G[W]\) that contains every edge of \(F\). Replacing each edge \(e=xy\in E(F)\) in this cycle by the
corresponding ear \(P_e\) yields a cycle containing every vertex of \(W\) of length
\[
|W|+\rho\ge L_0,
\]
which is a contradiction. Thus \(\rho\le L_0-|W|-1\).

 Since \(|U|=|V(G)|-|W|\), we have
\[
|U_{\mathrm{free}}|=|U|-\rho\ge |V(G)|-L_0+1.
\]
This proves (i).

\medskip

(ii) Fix \(u\in U_{\mathrm{free}}\) and set
\[
S_u=N_G(u)\cap(W\setminus W_{\mathrm{int}}).
\]
Then each vertex of \(S_u\) has degree at most \(1\) in \(F\), since \(W\setminus W_{\mathrm{int}}\) consists of the endpoints of nontrivial path components together with the isolated vertices.

We claim \(S_u\) lies in one component of \(F\). Otherwise, \(x,y\in S_u\) lie in distinct components. Since \(x,y\) each have degree at most \(1\), adding \(xy\) to \(F\) still yields a linear forest. The path \(xuy\) is a nontrivial \(W\)-ear internally disjoint from all ears in \(\mathcal P\) (as \(u\in U_{\mathrm{free}}\)), so \(\mathcal P\cup\{xuy\}\) has internal order \(\rho+1\), contradicting maximality. Thus \(S_u\) lies in one component of \(F\).

Since each component of \(F\) is a path and has at most two vertices of degree at most \(1\), we have \(|S_u|\le 2\), with equality only when the two vertices are the endpoints of that component. This proves (ii), and hence the lemma follows.
\end{proof}

We now prove the main result of this section. The following proposition provides the upper bound in the long-cycle range. Its proof uses the ear-forest structure established above to embed a large subtree of \(T_n\) in the red graph and then applies Hall's theorem to place the remaining leaves.

Let $N_{\mathrm{HP}}$ be an integer such that
Lemma~\ref{lem:dense-lks} holds with $\alpha=1/3$ for every
$N\ge N_{\mathrm{HP}}$.

\begin{proposition}\label{prop:long-cycle-range}
Let $n\ge5$ and suppose that $2n-1\ge N_{\mathrm{HP}}$. If $m$ is an odd integer satisfying
\[
            n+2\le m\le 2n-\left\lfloor\frac n2\right\rfloor,
\]
then every red--blue coloring of $K_{2n-1}$ contains either a red copy of $T_n$ or a blue $C_m$.
\end{proposition}

We need several auxiliary lemmas for the proof of Proposition~\ref{prop:long-cycle-range}. The first asserts that if two nonadjacent vertices have degree sum at least the order of the graph, then adding the edge between them preserves the existence of a cycle containing at least $k$ prescribed vertices.
\begin{lemma}[\v{C}ada, Flandrin, Li and Ryj\'a\v{c}ek~\cite{CadaEtAl2004}]
\label{lem:cada-closure}
If $G$ is a graph of order $N$ and $u,v$ are nonadjacent vertices with
$d_G(u)+d_G(v)\ge N$, then, for $G'=G+uv$, the following equivalence holds
for every $A\subseteq V(G)$ and every integer $1\le k\le |A|$:
$G$ has a cycle $C$ with $|V(C)\cap A|\ge k$ if and only if $G'$ does.
\end{lemma}

We shall use the following direct consequence, which preserves both a
prescribed vertex set and a lower bound on the cycle order.

\begin{lemma}
\label{lem:marked-closure}
If $G$ is a graph of order $N$ and $u,v$ are nonadjacent vertices with
$d_G(u)+d_G(v)\ge N$, then, for $G'=G+uv$, the following holds
for every $S\subseteq V(G)$ and every positive integer $\ell$:
if $G'$ has a cycle of length at least $\ell$ containing every vertex of $S$, then $G$ has such a cycle.
\end{lemma}

\begin{proof}
Let \(C\) be a cycle in \(G'\) of length at least \(\ell\) through \(S\).
Applying Lemma~\ref{lem:cada-closure} with \(A=V(C)\) and
\(k=|V(C)|\), we obtain a cycle \(C'\) in \(G\) with
\(|V(C')\cap V(C)|\ge |V(C)|\).
Thus \(C'\) contains all of \(V(C)\), and hence all of \(S\), and has
length at least \(|V(C)|\ge \ell\).
\end{proof}

Since Lemma~\ref{lem:marked-closure} preserves only a lower bound on the
cycle length, we also need to extract a cycle of the prescribed length
from a longer one. The next lemma shows that this is possible when the
longer cycle passes through an appropriate set of high-degree vertices.

\begin{lemma}\label{lem:exact-adjustment}
If $G$ is a graph of order $2n-1$ and $W\subseteq V(G)$ is an $n$-set such that
$d_G(w)\ge n$ for every $w\in W$, then, for every integer $m$ with
$3\le m\le2n-1$, the following holds:
whenever $G$ contains a cycle of length at least $m$ through $W$,
$G$ contains a cycle of length exactly $m$.
\end{lemma}

\begin{proof}
By hypothesis, $G$ contains a cycle through $W$. Choose such a cycle $Q$ of maximum order, and set
\[
\ell=|V(Q)|,\qquad
H=G[V(Q)],\qquad
Z=V(G)\setminus V(Q).
\]

Since $\ell\le 2n-1<2|W|$, the cycle $Q$ must contain two consecutive vertices of $W$. Otherwise, the successors on $Q$ of the vertices in $W$ would be distinct vertices outside $W$, implying $\ell\ge 2|W|=2n$, a contradiction. Let $uv$ be such an edge with $u,v\in W$.

Since $H=G[V(Q)]$, the cycle $Q$ is a Hamilton cycle of $H$. We claim that
\[
d_H(u)+d_H(v)\ge\ell+1.
\]
Suppose, to the contrary, that
$
d_H(u)+d_H(v)\le\ell.
$
Because $u,v\in W$, the degree hypothesis gives
$
d_G(u)+d_G(v)\ge2n.
$
Moreover,
\[
d_G(x)=d_H(x)+d_Z(x)
\qquad\text{for }x\in V(Q).
\]
It follows that
\[
\begin{aligned}
d_Z(u)+d_Z(v)
=d_G(u)+d_G(v)-d_H(u)-d_H(v)
\ge2n-\ell.
\end{aligned}
\]
Since
$
|Z|
=|V(G)|-|V(Q)|
=2n-1-\ell,
$
we obtain
$
d_Z(u)+d_Z(v)\ge |Z|+1.
$
Thus
\[
\bigl(N_G(u)\cap Z\bigr)
\cap
\bigl(N_G(v)\cap Z\bigr)
\ne\varnothing.
\]
Choose a common neighbor
$
z\in N_G(u)\cap N_G(v)\cap Z.
$

Replacing the edge $uv$ of $Q$ with the path $uzv$ gives a cycle
$Q'$ such that
\[
V(Q')=V(Q)\cup\{z\}.
\]
In particular, $Q'$ contains every vertex of $W$ and
$
|V(Q')|=\ell+1.
$
This contradicts the choice of $Q$ as a cycle of maximum order containing $W$. Hence
$
d_H(u)+d_H(v)\ge\ell+1.
$

The vertices $u$ and $v$ are consecutive on the Hamilton cycle $Q$ of
$H$.  Therefore Lemma~\ref{lem:bondy-cycle-edge} implies that $H$ is
pancyclic.  Since
$
3\le m\le\ell=|V(H)|,
$
the graph $H$, and hence $G$, contains $C_m$. This completes the proof.
\end{proof}

The next lemma provides a new tree-embedding tool: it embeds
a tree when the nonneighbors of each vertex on one side are
contained in a single class of a sufficiently fine partition of the
other side.

\begin{lemma}\label{lem:partition-defect-tree}
Let $H$ be a bipartite graph with parts $A$ and $B$, where $|A|=a$, and let
$\mathcal Q$ be a partition of $B$ into at least $a+1$ nonempty sets.
Suppose that, for every $u\in A$, there exists a member
$Q(u)\in\mathcal Q$ such that
\begin{equation}\label{eq:cover-condition}
B\subseteq N_H(u)\cup Q(u).
\end{equation}
If $T$ is a tree with bipartition $(X,Y)$ satisfying
$|X|=a>|Y|$,
then $H$ contains a copy of $T$ in which $X$ is embedded in $A$ and $Y$ is
embedded in $B$.
\end{lemma}

\begin{proof}
Fix a bijection $\phi\colon X\longrightarrow A$.
By hypothesis, for every $u\in A$ there is a class $Q(u)\in\mathcal Q$ such that \eqref{eq:cover-condition} holds. In particular, for each $x\in X$, let $Q_x=Q(\phi(x))$.

Define an auxiliary bipartite graph $\Gamma$ with parts $Y$ and
$\mathcal Q$ by setting, for each $y\in Y$,
\begin{equation}
N_\Gamma(y)
=
\mathcal Q\setminus\{Q_x:x\in N_T(y)\}.
\label{eq:class-neighborhood}
\end{equation}
Equivalently, $yQ\in E(\Gamma)$ if and only if $Q\ne Q_x$ for every
$x\in N_T(y)$. Thus $N_\Gamma(y)$ consists of the partition classes
from which the image of $y$ may be safely chosen: every vertex in such
a class is adjacent in $H$ to $\phi(x)$ for each $x\in N_T(y)$.

We claim that $\Gamma$ has a matching saturating $Y$. By Hall's theorem
\cite{Hall1935}, it suffices to show that
\[
|N_\Gamma(J)|\ge |J|
\]
for every $J\subseteq Y$. The case $J=\varnothing$ is trivial, so let
$\varnothing\ne J\subseteq Y$, and set $j=|J|$. Define
\[
\mathcal R_J=\bigcap_{y\in J}\{Q_x:x\in N_T(y)\}.
\]
It follows from \eqref{eq:class-neighborhood} that
$N_\Gamma(J)=\mathcal Q\setminus\mathcal R_J$. 

By the definition of $\mathcal R_J$, for every $Q\in\mathcal R_J$ and
$y\in J$, there exists $x\in N_T(y)$ such that $Q_x=Q$.
Choose one such edge $xy$ for each pair $(y,Q)$. These chosen edges are
distinct, since an edge $xy$ determines both $y$ and $Q_x$. Hence
\[
|\mathcal R_J|\cdot j\le e_T(J,N_T(J)).
\]
Since $T$ is a tree, the subgraph induced by $J\cup N_T(J)$ is a forest, so
\[
e_T(J,N_T(J))\le |J|+|N_T(J)|-1\le j+a-1.
\]
It follows that
\[
|\mathcal R_J| \le \frac{a+j-1}{j}.
\]
Since $j\le|Y|\le a-1$, we have $j(a+1-j)-(a+j-1)=(j-1)(a-j-1)\ge0$, and hence 
\[
|\mathcal R_J|\le\frac{a+j-1}{j}\le a+1-j.
\]
Therefore, using $|\mathcal Q|\ge a+1$ from the assumption, we obtain
\[
|N_\Gamma(J)|
=|\mathcal Q|-|\mathcal R_J|
\ge (a+1)-(a+1-j)
=j
=|J|.
\]
This proves the claim.

For each $y\in Y$, denote by $\widehat{Q}_y$ the member of $\mathcal Q$ matched to $y$.  The sets $\widehat{Q}_y$, $y\in Y$, are pairwise distinct. Moreover, since $y\widehat{Q}_y$ is an edge of $\Gamma$, the definition of $\Gamma$ gives
\begin{equation}
\widehat{Q}_y\ne Q_x
\qquad\text{for every }x\in N_T(y).
\label{eq:distinct-defect-classes}
\end{equation}

Since every member of $\mathcal Q$ is nonempty, choose a vertex
\[
z_y\in\widehat{Q}_y
\qquad\text{for each }y\in Y.
\]
Note that the classes $\widehat{Q}_y$ are pairwise disjoint, hence the vertices $z_y$ are distinct.

Define a map $\psi\colon V(T)\longrightarrow V(H)$ by
\[
\psi(v)=
\begin{cases}
\phi(v),&v\in X,\\
z_v,&v\in Y.
\end{cases}
\]
The restriction of $\psi$ to $X$ is injective because $\phi$ is a
bijection, and its restriction to $Y$ is injective because the vertices
$z_y$ are distinct.  Moreover,
\[
\psi(X)\subseteq A,\qquad
\psi(Y)\subseteq B,
\]
and $A\cap B=\varnothing$.  Hence $\psi$ is injective.

To verify that $\psi$ is an embedding, let $xy\in E(T)$, where
$x\in X$ and $y\in Y$.  Since $x\in N_T(y)$, the definition of
$\Gamma$ and the choice of $\widehat{Q}_y$ imply that
$\widehat{Q}_y\ne Q_x$ by \eqref{eq:distinct-defect-classes}.
The sets $\widehat{Q}_y$ and $Q_x$ are therefore distinct members of
the partition $\mathcal Q$.

 Since $z_y\in\widehat{Q}_y$, we have $z_y\notin Q_x$.
Applying \eqref{eq:cover-condition} with $u=\phi(x)$ gives $B\subseteq N_H(u)\cup Q_x$. Since $u\in A$ and $z_y\in B\setminus Q_x$, it follows that $z_y\in N_H(\phi(x))$.
Consequently,
$
\psi(x)\psi(y)=\phi(x)z_y\in E(H).
$
Thus $\psi$ is an embedding of $T$ into $H$ such that
$\psi(X)\subseteq A$ and $\psi(Y)\subseteq B$.
\end{proof}

With the necessary auxiliary lemmas in place, we now prove
Proposition~\ref{prop:long-cycle-range}.

\begin{proof}[Proof of Proposition~\ref{prop:long-cycle-range}]
Fix $n,m$ and $T_n$ as in Proposition~\ref{prop:long-cycle-range}. Suppose,
for a contradiction, that a red--blue coloring of $K_{2n-1}$ contains neither
a red copy of $T_n$ nor a blue $C_m$.

\medskip
\textbf{Step 1. Forming the blue clique and ear budget.}
Let $\R$ and $\B$ denote the red and blue spanning graphs, respectively.
We apply Lemma~\ref{lem:dense-lks} with $\alpha=1/3$ to the red graph on
$N=2n-1$ vertices. The bound $N\ge N_{\mathrm{HP}}$ holds by the proposition, and for $n\ge5$ we have
\[
\frac{N}{3}=\frac{2n-1}{3}<n-1<N.
\]
If at least $n$ vertices of $\R$ had degree at least $n-1$, then, since $n>N/2$, Lemma~\ref{lem:dense-lks} would imply that $\R$ contains a copy of $T_n$, a contradiction.
Consequently, at least $n$ of the $2n-1$ vertices have red degree at most
$n-2$, and hence blue degree at least
\[
 (2n-2)-(n-2)=n.
\]
Choose an $n$-set $W$ such that each vertex has blue degree at least $n$.

List the non-edges of $\B[W]$ (equivalently, the red edges in $W$) as $e_1,\ldots,e_t$, where $e_i=x_iy_i$, and define
\[
 \B_0=\B,
 \qquad
 \B_i=\B_{i-1}+e_i\quad(1\le i\le t).
\]
At step $i$, we add $e_i$ to the blue graph and let $\R_i=\overline{\B_i}$ be the resulting red graph.
Thus
\[
\B_0\subseteq \B_1\subseteq \cdots\subseteq \B_t,\qquad \R_0\supseteq \R_1\supseteq \cdots\supseteq \R_t.
\]

For each $1\le i\le t$, the edge $e_i=x_iy_i$ is not present
in $\B_{i-1}$.  Since $x_i,y_i\in W$ and every vertex of $W$ has
degree at least $n$ in $\B_0$, we have
\[
d_{\B_0}(x_i)\ge n,
\qquad
d_{\B_0}(y_i)\ge n.
\]
Therefore,
\[d_{\B_{i-1}}(x_i)+d_{\B_{i-1}}(y_i)
\ge d_{\B_0}(x_i)+d_{\B_0}(y_i)
\ge 2n
>|V(\B_{i-1})|.\]
Thus the nonadjacent vertices $x_i$ and $y_i$ satisfy the degree-sum
hypothesis of Lemma~\ref{lem:marked-closure} in $\B_{i-1}$.

Suppose that $\B_t$ contains a cycle of order at least $m$ containing every
vertex of $W$.
Applying Lemma~\ref{lem:marked-closure} successively for $i=t,t-1,\dots,1$
gives such a cycle in the original blue graph $\B_0$.
Applying Lemma~\ref{lem:exact-adjustment} to $\B_0$ yields a $C_m$ in $\B_0$, a contradiction.
Consequently, the augmented blue graph $\B_t$ has the following two properties:

\medskip
(1) the $n$-set $W$ is a blue clique;

\smallskip
(2) there is no cycle of length at least $m$ containing all vertices of $W$.

\medskip
We next convert the absence of a long blue cycle through \(W\) into
structural information needed to embed \(T_n\) in the red graph.

For the remainder of the proof, write \(\B=\B_t\) and
\(\R=\R_t\) for convenience, and let
\[
U=V(K_{2n-1})\setminus W.
\]
Thus \(|U|=n-1\).

\medskip
\textbf{Step 2. Extracting red structure from the long-cycle obstruction.}
We begin by applying Lemma~\ref{lem:max-ear-forest} to the augmented
blue graph \(\B\) with \(W\) as chosen and with parameter \(L_0=m\). The hypotheses are satisfied because \(W\) is a blue clique and \(\B\) has no cycle of length at least \(m\) through \(W\). The lemma gives a linear forest $F$ on $W$, a family of blue $W$-ears $\mathcal P=\{P_e:e\in E(F)\}$, and the sets
\[
\rho=\sum_{P\in\mathcal P}|V(P)\cap U|,\qquad
U_{\mathrm{free}}=U\setminus\bigcup_{P\in\mathcal P}V(P),\qquad
W_{\mathrm{int}}=\{w\in W:d_F(w)=2\}.
\]
The empty forest is allowed; in this case \(E(F)=\varnothing\), \(W_{\mathrm{int}}=\varnothing\), and the subsequent claims involving \(E(F)\) are vacuous.

By Lemma~\ref{lem:max-ear-forest}(i),
\begin{equation}\label{eq:U-size}
|U_{\mathrm{free}}|\ge (2n-1)-m+1=2n-m.
\end{equation}
By Lemma~\ref{lem:max-ear-forest}(ii), for every \(u\in U_{\mathrm{free}}\), the set \(N_\B(u)\cap(W\setminus W_{\mathrm{int}})\) is contained in a single path component of \(F\) and has size at most \(2\); if it has size \(2\), it is precisely the set of endpoints of that component.

For each edge \(e=xy\in E(F)\), let \(v_{e,x}\) and \(v_{e,y}\) denote the neighbors of \(x\) and \(y\), respectively, on the ear \(P_e\). We first record the following consequence of maximality.

\begin{claim}\label{clm:boundary-return}
If \(u\in U_{\mathrm{free}}\), \(x\in N_\B(u)\cap W_{\mathrm{int}}\), and \(e\in E(F)\) is incident with \(x\), then \(uv_{e,x}\) is a red edge.
\end{claim}

\begin{proof}
Suppose that \(uv_{e,x}\) were blue. The edge \(ux\) is blue by the definition of \(N_\B(u)\). Replace the edge \(xv_{e,x}\) of \(P_e\) with the path \(xuv_{e,x}\), and denote the resulting ear by \(P'_e\). Since \(u\in U_{\mathrm{free}}\), the path \(P'_e\) is internally disjoint from every ear in \(\mathcal P\setminus\{P_e\}\). Replacing \(P_e\) by \(P'_e\) preserves all endpoint pairs and hence the linear forest $F$, while increasing the total internal order from \(\rho\) to \(\rho+1\). This contradicts the maximality of \(\mathcal P\).
\end{proof}

Now define \(N_W(u)=N_\B(u)\cap W\) for each \(u\in U\), and for each \(u\in U_{\mathrm{free}}\), set
\begin{equation}
 \Lambda(u)=\bigl(W\setminus N_W(u)\bigr)
 \cup\{v_{e,x}:x\in N_W(u)\cap W_{\mathrm{int}},\ e\in E(F),\ x\in e\}.
 \label{eq:return-list}
\end{equation}
Every vertex of \(W\setminus N_W(u)\) is joined to \(u\) by a red edge, and Claim~\ref{clm:boundary-return} gives the same conclusion for every vertex in the second set in \eqref{eq:return-list}. Thus $\Lambda(u)\subseteq N_{\R}(u)$, i.e., every vertex of \(\Lambda(u)\) is a red neighbor of \(u\).

Let \(p=|E(F)|=|\mathcal P|\), and let \(c\) be the number of nontrivial components of $F$. Every member of $\mathcal P$ is a nontrivial $W$-ear and therefore has at least one internal vertex. Since the ears in $\mathcal P$ are pairwise internally vertex-disjoint,
\begin{equation}\label{eq-2}
\rho=\sum_{P\in\mathcal P}|V(P)\cap U|\ge |\mathcal P|=p.
\end{equation}

The preceding control of the blue neighborhoods yields the following
uniform estimate for the sets \(\Lambda(u)\), which will be used to verify the Hall condition for the sets of available images of the leaves of $T_n$.
\begin{claim}\label{clm:return-union}
For every nonempty set \(Z\subseteq U_{\mathrm{free}}\),
$\left|\bigcup_{u\in Z}\Lambda(u)\right|\ge n-2.$
\end{claim}

\begin{proof}
Let \(I=\bigcap_{u\in Z}N_W(u)\), and \(S=\bigcup_{u\in Z}\bigl(N_W(u)\cap W_{\mathrm{int}}\bigr)\). The set \(S\) collects all neighbors of vertices in \(Z\) that lie in \(W_{\mathrm{int}}\); it will be used to estimate the second part of the union in \(\Lambda(u)\).

By De Morgan's law,
\[
\bigcup_{u\in Z}\bigl(W\setminus N_W(u)\bigr)=W\setminus I,
\]
and since this set is contained in $\bigcup_{u\in Z}\Lambda(u)$, we have
\begin{equation}
\left|\bigcup_{u\in Z}\Lambda(u)\right|\ge |W\setminus I|=n-|I|. \label{eq:WsetminusI}
\end{equation}
Also, choosing any $u_0\in Z$, we have $I\subseteq N_W(u_0)$, so
$I\setminus W_{\mathrm{int}}\subseteq N_W(u_0)\setminus W_{\mathrm{int}},$
and therefore by Lemma~\ref{lem:max-ear-forest}(ii), 
\[
|I\setminus W_{\mathrm{int}}|\le 2.
\]

It remains to estimate the contribution from the second part of $\Lambda(u)$. For each $x\in S$, pick $u_x\in Z$ with $x\in N_W(u_x)\cap W_{\mathrm{int}}$, and define
\[
B_S=\{v_{e,x}:x\in S,\ e\in E(F),\ x\in e\}.
\]
By the definition of $\Lambda(u)$, we have $B_S\subseteq\bigcup_{u\in Z}\Lambda(u)$.

Since \(S\subseteq W_{\mathrm{int}}\), every vertex \(x\in S\) has
degree \(2\) in \(F\). Hence there are exactly \(2|S|\) pairs
\((x,e)\) such that \(x\in S\), \(e\in E(F)\), and \(x\in e\).
Each vertex of \(B_S\) can occur as \(v_{e,x}\) for at most two such
pairs, since distinct ears are internally vertex-disjoint and each
edge \(e\) has two endpoints. Therefore
\[
2|S|\le 2|B_S|,
\]
and hence
$|B_S|\ge |S|.$

Now $I\cap W_{\mathrm{int}}\subseteq S$: if $x\in I\cap W_{\mathrm{int}}$, then $x\in N_W(u)\cap W_{\mathrm{int}}$ for every $u\in Z$, so in particular $x\in S$. Thus
\[
|I|=|I\setminus W_{\mathrm{int}}|+|I\cap W_{\mathrm{int}}|
\le 2+|S|.
\]

Finally, using the fact that $W\setminus I$ and $B_S$ are disjoint, together with \eqref{eq:WsetminusI} and the estimate $|B_S|\ge |S|$, we obtain
\[
\begin{aligned}
\left|\bigcup_{u\in Z}\Lambda(u)\right|
&\ge |W\setminus I|+|B_S| \\
&\ge (n-|I|)+|S| \\
&= n-(|I\setminus W_{\mathrm{int}}|+|I\cap W_{\mathrm{int}}|)+|S|\\
&\ge n-2.
\end{aligned}
\]
This proves the claim.
\end{proof}

\textbf{Step 3. Embedding the tree.}
We now use the structure established in Step 2 to embed the tree $T_n$.

Let $(X,Y)$ be the bipartition of $T_n$, with $|X|=a\le b=|Y|$ and $a+b=n$.
Let $Y_1$ be the set of leaves in $Y$, let $\lambda=|Y_1|$, and set $Y_0=Y\setminus Y_1$ with $b_0=|Y_0|$.  Since the degree sum over $Y$ counts each edge of $T_n$ exactly once, and each vertex of $Y_0$ has degree at least $2$, we have
\[
n-1=\sum_{y\in Y}d_{T_n}(y)\ge \lambda+2(b-\lambda)=2b-\lambda,
\]
which, together with $a+b=n$ and $b_0=b-\lambda$, yields
\begin{equation}
\lambda\ge b-a+1,\qquad b_0\le a-1. \label{eq:leaves}
\end{equation}

If $a=1$, then $T_n$ is the star $K_{1,n-1}$.  Since the original coloring contains no red copy of this star, every red degree is at most $n-2$, so the original blue graph has minimum degree at least $n>(2n-1)/2$.  By Corollary~\ref{lem:bondy}, it is pancyclic, contradicting the absence of a blue $C_m$.

Henceforth, assume $a\ge2$. Let $T_0=T_n-Y_1$.  Since every vertex of $Y_1$ is a leaf, $T_0$ is a tree with bipartition $(X,Y_0)$; moreover, by \eqref{eq:leaves}, $|X|=a>b_0=|Y_0|$.

Set 
$
W_0 = W\setminus W_{\mathrm{int}}.
$
Combining \eqref{eq:U-size} with
$m\le 2n-\lfloor n/2\rfloor$ and
$a\le\lfloor n/2\rfloor$, we obtain
\begin{equation}
|U_{\mathrm{free}}|
\ge 2n-m
\ge \left\lfloor\frac n2\right\rfloor
\ge a.
\label{eq:U-embedding-size}
\end{equation}

We next define the partition $\mathcal Q$ of $W_0$ required in Lemma~\ref{lem:partition-defect-tree}. Recall the linear forest $F$ on $W$ defined by the endpoint pairs of the ears in $\mathcal P$. Place the two endvertices of each nontrivial component of $F$ in one class, and each isolated vertex in a singleton class.

Since a nontrivial path component with \(r\) edges has \(r+1\)
vertices, the \(c\) nontrivial components of \(F\) contain \(p+c\)
vertices altogether. Hence \(F\) has \(n-p-c\) isolated vertices.
By the definition of \(\mathcal Q\), there is one two-vertex class for
each nontrivial component and one singleton class for each isolated
vertex. Therefore,
\[
|\mathcal Q|=c+(n-p-c)=n-p.
\]

By \eqref{eq-2}, we have \(p\le \rho\). Since the ears in
\(\mathcal P\) are pairwise internally vertex-disjoint, their interiors
contain \(\rho\) vertices in total. As \(|U|=n-1\), the definition of
\(U_{\mathrm{free}}\) gives
$
|U_{\mathrm{free}}|
=|U|-\rho
=n-1-\rho.
$
Consequently, by \eqref{eq:U-embedding-size},
\[
|\mathcal Q|=n-p
\ge n-\rho
=|U_{\mathrm{free}}|+1
\ge a+1.
\]

Choose a set $U_{\mathrm{free}}'\subseteq U_{\mathrm{free}}$ of order $a$.  For each $u\in U_{\mathrm{free}}'$, the set of its nonneighbors in the red bipartite graph between $U_{\mathrm{free}}'$ and $W_0$ is $N_\B(u)\cap W_0=N_W(u)\setminus W_{\mathrm{int}}$, which is contained in one member of $\mathcal Q$ by Lemma~\ref{lem:max-ear-forest}(ii).  Thus, by Lemma~\ref{lem:partition-defect-tree}, the red bipartite graph with parts $U_{\mathrm{free}}'$ and $W_0$ contains a copy of $T_0$.  Let $\phi$ be an isomorphism from $T_0$ to this copy such that
\[
\phi(X)\subseteq U_{\mathrm{free}}',\qquad \phi(Y_0)\subseteq W_0.
\]

It remains to embed the leaves in $Y_1$.  For each $y\in Y_1$, let $u_y\in U_{\mathrm{free}}$ be the image of the unique neighbor of $y$ in $X$, and define
\[
M_y=\Lambda(u_y)\setminus \phi(Y_0),
\]
the set of available candidates for $y$.

By the observation after \eqref{eq:return-list}, every vertex of
$\Lambda(u_y)$ is a red neighbor of $u_y$. Moreover,
$\Lambda(u_y)\cap U_{\mathrm{free}}'=\varnothing$: indeed, the first
part of $\Lambda(u_y)$ lies in $W$, while every vertex in its second
part lies internally on an ear in $\mathcal P$ and hence outside
$U_{\mathrm{free}}$. Therefore, after removing $\phi(Y_0)$, every
vertex of $M_y$ is unused by the current embedding.

We now verify Hall's condition for the family $\{M_y:y\in Y_1\}$.
For a nonempty set $J\subseteq Y_1$, let
$U_J=\{u_y:y\in J\}\subseteq U_{\mathrm{free}}$. Then
$
\bigcup_{y\in J}\Lambda(u_y)
=
\bigcup_{u\in U_J}\Lambda(u),
$
and hence Claim~\ref{clm:return-union} gives
\[
\left|\bigcup_{y\in J}\Lambda(u_y)\right|\ge n-2.
\]
Removing the $b_0$ vertices of $\phi(Y_0)$ can reduce this union by at
most $b_0$, so
\[
\left|\bigcup_{y\in J}M_y\right|
\ge
\left|\bigcup_{y\in J}\Lambda(u_y)\right|-|\phi(Y_0)|
\ge n-2-b_0.
\]
Since $a+b=n$, $b_0=b-\lambda$, and $a\ge2$, it follows that
\[
\left|\bigcup_{y\in J}M_y\right|
\ge n-2-b_0
=a+\lambda-2
\ge\lambda
=|Y_1|
\ge|J|.
\]
Thus Hall's condition holds for every $J\subseteq Y_1$. By Hall's
theorem \cite{Hall1935}, the family $\{M_y:y\in Y_1\}$ has a system of
distinct representatives. Using these representatives as the images
of the leaves extends the current embedding to a red copy of $T_n$ in
$\R_t$. Since $\R_t\subseteq\R_0$, this copy is red in the original
coloring, contradicting our assumption and completing the proof of Proposition~\ref{prop:long-cycle-range}.
\end{proof}

\section{Proof of Theorem~\ref{thm:main}}\label{sec:main-proof}

Let $m\ge5$ be any odd integer, and let $n$ be any integer satisfying $n\ge\max\{C,\lceil(2m-1)/3\rceil\}$, 
where $C=\max\{5,\lceil(N_{\mathrm{HP}}+1)/2\rceil\}$. Set $N=2n-1$.
Consider a red--blue coloring of $K_N$ with no red copy of $T_n$; we will show that it contains a blue $C_m$.

\smallskip
(1) $m\le n+1$: Since $n\ge C\ge5$ and $N=2n-1$, we have $N/3<n-1<N$, so Lemma~\ref{lem:dense-lks} applies with $\alpha=1/3$ and $k=n-1$. If at least $n$ vertices have red degree at least $n-1$, the lemma yields a red copy of $T_n$, a contradiction. Thus fewer than $n$ vertices have red degree at least $n-1$; equivalently, at least $n$ vertices have red degree at most $n-2$. Each such vertex has blue degree at least $(2n-2)-(n-2)=n$. Applying Lemma~\ref{lem:median-pancyclic} with $q=n$ gives blue cycles of every length $3\le s\le n+1$, and hence, in particular, a blue $C_m$. This completes the short-cycle case.

\smallskip
(2) $m\ge n+2$: The condition $n\ge\lceil(2m-1)/3\rceil$ implies $3n\ge2m-1$, and therefore
\[
m\le \left\lfloor\frac{3n+1}{2}\right\rfloor
=2n-\left\lfloor\frac n2\right\rfloor.
\]
Together with $n\ge C\ge5$ and $N=2n-1\ge N_{\mathrm{HP}}$ (as $n\ge C$), all hypotheses of Proposition~\ref{prop:long-cycle-range} are satisfied. That proposition yields either a red $T_n$ or a blue $C_m$.

\smallskip
Therefore, under the assumptions of the theorem, every such coloring contains a red $T_n$ or a blue $C_m$, proving $R(T_n,C_m)\le2n-1$. Together with the lower bound $R(T_n,C_m)\ge2n-1$ from \eqref{eq:lower}, the theorem follows. \hfill$\square$

\bigskip
\noindent \textbf{Concluding remarks.}
In our proof, the constant \(C\) enters only through the large-order
threshold in the dense Loebl--Koml\'os--S\'os theorem. Since that
theorem guarantees only the existence of such a threshold and does not
give an explicit value of \(N_{\mathrm{HP}}\), the present argument
does not give an explicit value of \(C\). It therefore leaves open
the determination of \(f(m)\) for finitely many smaller odd values of
\(m\). Determining these remaining values and describing the extremal
colorings near the threshold remain natural problems.

For all sufficiently large odd \(m\), we determine the exact threshold for trees:
\(f(m)=\left\lceil\frac{2m-1}{3}\right\rceil.\)
This sharpens the previously known linear bounds to an exact result.
Fan and Lin~\cite{FanLin2025} obtained linear threshold results in a
broader sparse-graph setting. It is natural to ask whether comparably
sharp thresholds can be determined for other natural classes of sparse
graphs, such as unicyclic graphs and, more generally, connected graphs
with bounded cyclomatic number. Recall that the cyclomatic number of a
graph \(G\) is \(\beta(G)=|E(G)|-|V(G)|+c(G),\)
where \(c(G)\) denotes the number of components of \(G\).
Such an extension would likely require additional ideas beyond the
ear-forest argument used here. In particular, the final Hall-based
embedding step is tailored to attaching leaves of a tree and does not
directly control the additional edges that create cycles.

\medskip
\noindent\textbf{Declaration of generative AI use.}
The authors used generative AI tools solely for language editing and
proofreading. All mathematical content and arguments were developed and
verified by the authors, who take full responsibility for the manuscript.

\bibliographystyle{plain}

\end{document}